\documentclass[10pt,reqno,a4paper]{amsart}
\usepackage{amsmath,amsfonts,amsthm,amssymb,url,graphicx,algorithm,algorithmicx,tikz,float,hyperref,cleveref}
\usepackage{amsmath}
\usepackage{bigints}
\usepackage{fontawesome}
\usepackage{tikz-cd}
\usepackage{mathrsfs}
\usepackage{amssymb}
\usepackage{amsbsy}
\usepackage{dsfont}
\usepackage{amsthm}
\usepackage{graphicx}
\usepackage{xypic}
\usepackage{CJKutf8}
\usepackage{nccmath}
\usepackage{skak}
\usepackage{enumerate}
\usepackage{commath} 
\usepackage{tikz}
\usepackage{relsize}
\usepackage{faktor}
\usepackage{pst-plot}
\usepackage{float}
\usepackage[shortlabels]{enumitem}
\usepackage[showonlyrefs]{mathtools}

\usepackage[all]{xy}

\usepackage{thmtools}
\usepackage{thm-restate}

\usetikzlibrary{decorations.markings}
\usetikzlibrary{arrows}
\usetikzlibrary{arrows.meta}
\usepackage{etoolbox}

\usepackage[
  backend=bibtex,
  style=numeric,
  firstinits=true,
  doi=false, isbn=false, eprint=false,
  sorting=anyt, maxbibnames=99
]{biblatex}
\nocite{*}

\DeclareFieldFormat[article]{volume}{\mkbibbold{#1}} 
\DeclareFieldFormat*{title}{\textit{#1}} 
\DeclareFieldFormat[article]{number}{no. #1} 
\DeclareFieldFormat{pages}{#1} 
\DeclareFieldFormat[article]{journaltitle}{#1} 

\renewbibmacro{in:}{}
\DeclareDelimFormat{nametitledelim}{\addcomma\space} 
\DeclareDelimFormat{titleyeardelim}{\addcomma\space} 

\DeclareBibliographyDriver{article}{%
  \usebibmacro{bibindex}%
  \usebibmacro{begentry}%
  \usebibmacro{author}%
  \setunit{\nametitledelim}\newblock 
  \usebibmacro{title}%
  \setunit{\titleyeardelim}\newblock 
  \usebibmacro{journal}%
  \setunit{\addcomma\space}%
  \printfield{volume}%
  \setunit{\addspace}%
  \printtext[parens]{\usebibmacro{date}}
  \setunit{\addcomma\space}%
  \printfield{number}%
  \setunit{\addcomma\space}%
  \usebibmacro{note+pages}%
  \usebibmacro{finentry}}

\DeclareFieldFormat{url}{arXiv: \url{#1}}

\DeclareBibliographyDriver{misc}{%
  \usebibmacro{bibindex}%
  \usebibmacro{begentry}%
  \usebibmacro{author}%
  \setunit{\nametitledelim}\newblock
  \usebibmacro{title}%
  \setunit{\addspace}%
  \printtext[parens]{\usebibmacro{date}}. 
  \setunit{\addspace}%
  \printfield{url} 
  \usebibmacro{finentry}%
}

\usepackage{listings}

\usepackage{lipsum}

\hypersetup{
           breaklinks=true,   
           colorlinks=true,   
           pdfusetitle=true,  
              linkcolor=black,citecolor=black , urlcolor=black            
        }

\makeatletter
\newcommand{\pushright}[1]{\ifmeasuring@#1\else\omit$\displaystyle#1$\ignorespaces\fi}
\makeatother
\makeatletter
\def\moverlay{\mathpalette\mov@rlay}
\def\mov@rlay#1#2{\leavevmode\vtop{%
   \baselineskip\z@skip \lineskiplimit-\maxdimen
   \ialign{\hfil$\m@th#1##$\hfil\cr#2\crcr}}}
\newcommand{\charfusion}[3][\mathord]{
    #1{\ifx#1\mathop\vphantom{#2}\fi
        \mathpalette\mov@rlay{#2\cr#3}
      }
    \ifx#1\mathop\expandafter\displaylimits\fi}
\makeatother

\definecolor{indigo}{RGB}{75, 0, 130}
\definecolor{midnightblue}{RGB}{25, 25, 112}

\makeatletter

\AtBeginDocument{%
 \@ifpackageloaded{refcheck}{%
 \@ifundefined{hyperref}{}{%
 \let\T@ref@orig\T@ref
 \def\T@ref#1{\T@ref@orig{#1}\wrtusdrf{#1}}%
 \let\@refstar@orig\@refstar
 \def\@refstar#1{\@refstar@orig{#1}\wrtusdrf{#1}}%
 \DeclareRobustCommand\ref{\@ifstar\@refstar\T@ref}%
 }%
 }{}%
}
\makeatother

\usepackage{pgfplots}

\pgfplotsset{compat=1.6}

\pgfplotsset{soldot/.style={color=blue,only marks,mark=*}} \pgfplotsset{holdot/.style={color=blue,fill=white,only marks,mark=*}}

\def\XXint#1#2#3{{\setbox0=\hbox{$#1{#2#3}{\int}$ }
\vcenter{\hbox{$#2#3$ }}\kern-.6\wd0}}

\usetikzlibrary{arrows,shapes,positioning}
\usetikzlibrary{decorations.markings}
\tikzstyle arrowstyle=[scale=1]
\tikzstyle directed=[postaction={decorate,decoration={markings,
    mark=at position .65 with {\arrow[arrowstyle]{stealth}}}}]
\tikzstyle reverse directed=[postaction={decorate,decoration={markings,
    mark=at position .65 with {\arrowreversed[arrowstyle]{stealth};}}}]

\DeclareSymbolFont{extraup}{U}{zavm}{m}{n}
\DeclareMathSymbol{\varheart}{\mathalpha}{extraup}{86}
\DeclareMathSymbol{\vardiamond}{\mathalpha}{extraup}{87}

\usepackage{color}
\usepackage{mathrsfs}
\usepackage{yfonts}
\usepackage{multirow}
\usepackage{bbm}

\usepackage[toc]{glossaries}
\usepackage{glossary-mcols}
\usepackage{chngcntr}
\usepackage{apptools}
\AtAppendix{\counterwithin{lem}{chapter}}
\AtAppendix{\counterwithin{prop}{chapter}}
\AtAppendix{\counterwithin{theorem}{chapter}}
\AtAppendix{\counterwithin{definition}{chapter}}
\AtAppendix{\counterwithin{cortheo}{chapter}}
\AtAppendix{\counterwithin{corprop}{chapter}}

\makeglossaries

\usepackage[utf8]{inputenc}
\usepackage[english]{babel}

\newtheorem{genericthm}{Dummy}[section] 
\numberwithin{equation}{section}

\declaretheorem[name=Theorem, numberlike=genericthm]{theorem}
\declaretheorem[name=Lemma, numberlike=genericthm]{lem}

\declaretheorem[name=Example, style=definition, numberlike=genericthm]{exmp}

\declaretheorem[name=Remark, style=remark, numbered=no]{remark}

\newcommand{\mirror}[1]{\overleftarrow{#1}}

\makeatletter
\newcommand*{\defeq}{\mathrel{\rlap{%
                     \raisebox{0.24ex}{$\m@th\cdot$}}%
                     \raisebox{-0.24ex}{$\m@th\cdot$}}%
                     =}
\makeatother

\renewcommand{\1}{\mathds{1}}
\newcommand{\lt}{\ensuremath <}
\newcommand{\gt}{\ensuremath >}

\renewcommand{\le}{\leqslant}
\renewcommand{\ge}{\geqslant}

\renewcommand{\d}{\dif }

\newcommand{\ve}{\varepsilon}

\renewcommand{\ss}{\subseteq}

\newcommand{\al}{\alpha}

\newcommand{\la}{\lambda}

\newcommand{\be}{\beta}
\newcommand{\de}{\delta}

\newcommand{\ze}{\zeta}
\newcommand{\f}{\frac}
\newcommand{\mf}{\mfrac}

\newcommand{\cs}{\mathscr}
\newcommand{\mc}{\mathcal}

\newcommand{\bb}{\mathbb}
\newcommand{\dt}{\dif t}

\newcommand{\dx}{\dif x}

\DeclarePairedDelimiterX{\inn}[2]{\langle}{\rangle}{#1, #2}

\DeclareMathOperator{\me}{e}

\DeclareSymbolFont{eulargesymbols}{U}{zeuex}{m}{n}
\DeclareMathSymbol{\intop}{\mathop}{eulargesymbols}{"52}

\def\d{{\rm d}}

\title[Powerfree integers and Fourier bounds ]{Powerfree integers and Fourier bounds }

\author{Sebastián Carrillo Santana}
\address{Mathematics Institute, Utrecht University,  Hans Freudenthalgebouw, Budapestlaan 6, 3584 CD Utrecht, Netherlands}
\email{s.carrillosantana@uu.nl}
\begin{document}
\begin{abstract}
We develop a general approach for showing when a set of integers $\mathscr{A}$ has infinitely many $k^{th}$ powerfree numbers without relying on equidistribution estimates for $\mathscr{A}$. In particular, we show that if the Fourier transform of $\mathscr{A}$ satisfies certain $L^{\infty}$ and $L^{1}$ bounds, and is also ``decreasing'' in some sense, then $\mathscr{A}$ contains infinitely many $k^{th}$ powerfree numbers. We then use this method to show that there are infinitely many cubefree palindromes in base $b\ge 1100$, and in the process we obtain new $L^{1}$ bounds for the Fourier transform of the set of palindromes. We also show that there are infinitely many squarefree integers such that its reverse is also squarefree in any base $b\ge 2$. Moreover, we show that there are infinitely many squarefree integers with a missing digit in base $b\ge 5$, and infinitely many such cubefree integers in base $b\ge 3$. 
\end{abstract}

\maketitle

\section{Introduction}

In the past few decades, there has been significant interest in sets of integers defined by arithmetic properties of their digits. We mention for example the work of Mauduit and Rivat \cite{MauduitRivat} on the sum of digits of primes, the work of Bourgain \cite{Bourgain1,Bourgain2} and Swaenepoel \cite{Swaenepoel} on primes with prescribed digits, the work of Maynard \cite{MaynardDigitsOfPrimes, Maynard} on primes with missing digits, and more recently the work of Dartyge, Martin, Rivat, Shparlinkski, and Swaenepoel \cite{Cecile} on reversible primes.

The purpose of this article is to develop a general approach for showing when a set of integers has infinitely many $k^{th}$ powerfree numbers. This method is primarily based on obtaining certain bounds for the Fourier transform of a set of integers. This approach applies to many sets of integers defined by arithmetic properties of their digits, because in most of such cases, the Fourier transform has a nice multiplicative structure.  

Given a set of integers $\cs{A}$ and a positive $x\in\bb{R}$, let $\cs{A}(x)\defeq \cs{A}\cap [1,x]$. We define the Fourier transform of $\cs{A}(x)$ as the function $S_x\colon [0,1]\to \bb{C}$ defined by
\[
S_x(t)\defeq \sum_{n\in\cs{A}(x)}\me(nt),
\]
where $\me(z)=\me^{2\pi i z}$. It is also convenient to define a normalized version:
\[
F_x(t)\defeq \frac{1}{\#\cs{A}(x)}|S_x(t)|.
\]
Of course $F_x(t)$ depends on the set $\cs{A}$, but for ease of notation we avoid writing such dependence; from the context it should be clear what $F_x(t)$ means, depending on the set $\cs{A}$ we are interested in. We are ready to state our main result.  
\begin{restatable}{theorem}{main}\label{th: Main Theorem}
  Suppose that $\cs{A}$ is a set of integers such that its normalized Fourier transform $F_x$ satisfies the following properties:
  \begin{enumerate}[{\rm (i)}]
    \item  An $L^{\infty}$ type estimate of the form
  \begin{equation}\label{eq: L^infinity estimate}
  F_x\Big(\frac{a}{d}\Big) \ll \exp{\Big(-c\frac{\log{x}}{\log{d}}\Big)} 
  \end{equation}
  for some absolute constant $c\gt 0$ and any integers $a,d$ with $0\lt a\lt d$ and $2\le d\le (\log{x})^B$ for some $B\gt 0$.
    
    \item $L^1$ bounds of the form
  \begin{equation}\label{eq: L^1 bounds}
  \int_{0}^{1}F_x(t)\dt \ll \frac{1}{x^{\f{1}{k}+\de}}\quad\mbox{and}\quad \int_{0}^{1}|F_x'(t)|\dt \ll x^{1-\f{1}{k}-\de}
  \end{equation}
  for some positive integer $k \ge 2$ and some $\de\gt 0$.
  \item There is some integer $b\ge 2$ such that for any $x,y\in \{b^m \; : \; m\in\bb{Z}^{+}\}$ with $x\le y$, we have 
  \begin{equation}\label{eq: Fourier transform decreasing in x}
  F_y(t)\ll F_x(t).
  \end{equation}
  \end{enumerate}
  Then $\cs{A}$ contains infinitely many $k^{th}$ powerfree integers and 
   \[
  \sum_{\substack{n\in\cs{A}(x) \\ n \text{ is } k^{\text{th}} \text{ powerfree}}}1\sim \frac{\#\cs{A}(x)}{\ze(k)} 
  \]
  as $x\to\infty$.
\end{restatable}
\begin{remark}
Observe that if $\cs{A}(x)\asymp x^{\theta}$ for some $\theta\in (0,1]$, then by Parseval's identity, 
\[
\int_{0}^{1}|S_x(t)|^2\dt=\sum_{n\in\cs{A}(x)} 1\ll x^{\theta}. 
\]
Therefore, by the Cauchy Schwartz inequality, 
\[
\int_{0}^{1}|F_x(t)|\dt \ll \frac{1}{x^{\f{\theta}{2}}},
\]
and we have the same bound for $\norm{F_x'}$ but multiplied by $x$. This shows that \eqref{eq: L^1 bounds} always holds for some $k$ as long as $\cs{A}(x)\asymp x^{\theta}$. Of course, if the Fourier transform of $\cs{A}$ has considerable cancellation, in some cases, we expect to have a much stronger bound of the form 
\[
\int_{0}^{1}|F_x(t)|\dt\ll \frac{1}{x^{\theta-\ve}}
\]
for any $\ve\gt 0$, but this is usually difficult to show. The reason for expecting this is because, by the work of McGehee, Pigno, and Smith \cite{L1NormExponentialSums}, and independently by Konyagin \cite{KonyaginL1}, we know that 
\begin{equation}\label{eq: LowerBoundL1}
\frac{\log{x}}{x^{\theta}}\asymp\frac{\log{x}}{\#\cs{A}(x)}\ll \int_{0}^{1}|F_x(t)|\dt.
\end{equation}
In particular, if $\theta\le \f12$, our methods are not strong enough to prove the existence of infinitely many squarefree integers because \eqref{eq: L^1 bounds} would be impossible. 
\end{remark}
The main philosophy of Theorem \ref{th: Main Theorem} is that, if we have good enough $L^{1}$ and $L^{\infty}$ bounds, and if we have some kind of decreasing property for $F_x(t)$, then we can show that our set has infinitely many $k^{th}$ powerfree integers for some $k$ that depends on how strong our $L^{1}$ estimates are. There are many scenarios in which the Fourier transform $F_x(t)$ does not exactly satisfy the same properties as in Theorem \ref{th: Main Theorem}, but it nevertheless satisfies some analogous properties. Consequently, we should think of Theorem \ref{th: Main Theorem} more as a general method rather than a standalone theorem. One of the main features of this framework is that we don't rely on equidistribution estimates. In the following subsections, we state some applications of this method.
\subsection{Cubefree palindromes}
Banks, Hart, and Sakata \cite{Banks} showed that almost all palindromes are composite. Afterwards, Col \cite{Col} improved upon these results by obtaining an upper bound of the right order of magnitude for the number of palindromes less than or equal to $x$. The author did this by obtaining equidistribution estimates for palindromes in arithmetic progressions. Recently, Tuxanidy and Panario \cite{TuxanidyPanario} improved upon Col's results by extending the level of distribution of palindromes in arithmetic progressions to moduli up to $x^{\f15-\de}$.

Let $\cs{P}_b$ denote the set of nonnegative palindromes in base $b\ge 2$. Using equidistribution results for square moduli for palindromes from Tuxanidy and Panario, Chourasiya and Johnston \cite{ShashiDaniel} proved that there are infinitely many $4^{th}$ powerfree palindromes in any base $b\ge 2$, and they gave the following asymptotic:
\begin{theorem}[\cite{ShashiDaniel},Theorem 1.6]\label{th: ShashiDaniel}
  Let
  \[
  \cs{P}_b^{*}\defeq \{n\in\cs{P}_b \; : \; (n,b^3-b)=1\}.
  \]
  Then,
    \[
  \sum_{\substack{n\in\cs{P}_b^{*}(x) \\ n \text{ is } 4^{th}\text{ powerfree}}}1\sim \frac{\#\cs{P}_b^{*}(x)}{\zeta(4)}\prod_{p|b^3-b}\Big(1-\frac{1}{p^4}\Big)^{-1}.
  \]
\end{theorem}
By obtaining new $L^{1}$ estimates for the Fourier transform of the set of palindromes, we showed independently and almost simultaneously, that for $b\ge 1100$, $\cs{P}_b^{*}$ contains infinitely many cubefree integers:
\begin{restatable}{theorem}{cubefree}\label{th: Cubefree}
  For $b\ge 1100$, $\cs{P}_b^{*}$ contains infinitely many cubefree palindromes, and 
  \[
  \sum_{\substack{n\in\cs{P}_b^{*}(x) \\ n \text{ is cubefree}}}1\sim \frac{\#\cs{P}_b^{*}(x)}{\zeta(3)}\prod_{p|b^3-b}\Big(1-\frac{1}{p^3}\Big)^{-1}.
  \]
\end{restatable}
\begin{remark}
Of course the result also holds if we replace the word cubefree by $k^{th}$ powerfree for $k\ge 3$, but we work with cubefree for the sake of simplicity. 
\end{remark}
Subsequently, Chourasiya and Johnston \cite{ShashiDaniel} obtained Theorem \ref{th: Cubefree} for all bases $b\ge 2$, again using equidistribution estimates from Tuxanidy and Panario \cite{TuxanidyPanario}, and a Brun–Titchmarsh style result due to Banks
and Shparlinski \cite{BanksShparlinski}. In this paper, we give a proof of Theorem \ref{th: Cubefree}.
\subsection{Squarefree reversible integers}
Given an integer $n$ written in base $b\ge 2$ as 
\[
n=\sum_{j=0}^{k}n_j b^j
\]
for some $n_j\in \{0,\ldots, b-1\}$, we define the \emph{reverse} of $n$ in base $b$ as 
\[
\mirror{n}_b\defeq\sum_{j=0}^{k}n_j b^{k-1-j}.
\]
Recently, Dartyge, Martin, Rivat, Shparlinkski, and Swaenepoel \cite{Cecile} showed that there are infinitely many squarefree integers $n$ such that $\mirror{n}_b$ is squarefree in base $b=2$:
\begin{theorem}\label{th: ReversibleSquarefreeBase2}
  Let 
  \[
  \cs{D}_l\defeq \{ n\in \bb{Z} \; : \; 2^{l-1}\le n \lt 2^l \mbox{ and } n \mbox{ is odd}\},
  \]
  and let 
  \[
  Q(l)=\# \{ n\in \cs{D}_l \; : \; \mu^2(n)=\mu^2(\mirror{n}_2)=1\}.
  \]
  Then, 
  \[
  Q(l)\sim \frac{66}{\pi^4}\#\cs{D}_l.
  \]
\end{theorem}
Using results for the Fourier transform associated to reversible integers from Bhowmik and Suzuki \cite{GautamiYuta}, recently improved by Dartyge, Rivat, and Swaenepoel \cite{DartygeRivatSwaenepoel}, and independently by Bhowmik and Suzuki \cite{GautamiYuta2}, we generalize Theorem \ref{th: ReversibleSquarefreeBase2} for all $b\ge 2$:
\begin{theorem}\label{th: n and rev(n) squarefree}
Let $b\ge 2$, and let 
\[
\cs{H}_b\defeq \{n\in\bb{Z}_{\ge 1} \; : \; (n,b^3-b)=1, \; (\mirror{n}_b,b^3-b)=1\}.
\]
Then, 
\[
\sum_{\substack{n\in\cs{H}_b(x) \\ n,\mirror{n}_b \text{ are squarefree}}}1\sim \frac{\#\cs{H}_b(x)}{\zeta(2)^2}\prod_{p|b^3-b}\Big(1-\frac{1}{p^2}\Big)^{-2}.
\]
\end{theorem}
\subsection{Powerfree integers with missing digits}
Filaseta and Konyagin \cite{FilasetaKonyagin} showed that there are infinitely many squarefree integers in base $b=3,4,5$ consisting only of the digits $0$ and $1$. Afterwards, using equidistribution estimates, Erd\H os, Mauduit, and S\'ark\"ozy \cite{ErdosMauduitSarkozy} obtained an asymptotic for the number of $k^{th}$ powerfree integers in base $b$ with $q$ excluded digits:
\begin{theorem}[\cite{ErdosMauduitSarkozy}, Theorem 4]\label{th: ErdosMauduitSarkozy}
  Let $b\ge 3$, and $2\le q\le b-1$ be integers, let $\cs{Q}\ss \{0,\ldots, b-1\}$ with $\#\cs{Q}=q$, and let 
   \[
   \cs{W}_b\defeq \Big\{ \sum_{j=0}^{m}n_j b^j \; : \; n_j\in \{0,\ldots ,b-1\}\setminus \cs{Q}, \, m\ge 0\Big\}.
   \]
   Assume further that for some positive integer $k$, 
  \[
  k\gt \frac{\log{b}}{\log{(b-q)}}.
  \]
  Then,
   \[
  \sum_{\substack{n\in\cs{W}_b(x) \\ n \text{ is } k^{th} \text{ powerfree}\\ (n,b^2-b)=1}}1\sim \frac{\#\cs{W}_b(x)}{\zeta(k)}\prod_{p|b^2-b}\Big(1-\frac{1}{p^k}\Big)^{-1}.
  \]
\end{theorem}
Observe that the above result proves Filaseta and Konyagin's result for $b=3$. However, it is still an open question to show whether or not for $b\ge 6$ there are infinitely many squarefree integers such that every digit in base $b$ is 0 or 1.

In the particular case where we exclude one digit, say $\cs{Q}=\{a_0\}$, Maynard \cite{Maynard} showed that for bases $b\ge 10$, the set $\cs{W}_b$ contains infinitely many primes:
\begin{theorem}[\cite{Maynard}, Theorem 1.1]\label{th: MaynardMainResult}
  Given an integer $b\ge 10$, and some $a_0\in \{0,\ldots , b-1\}$, let 
  \[
\cs{M}_b\defeq \Big\{ \sum_{j=0}^{k}n_j b^j \; : \; n_j\in \{0,\ldots ,b-1\}\setminus \{a_0\}, \, k\ge 0\Big\}
\]
be the set of integers in base $b$ without $a_0$ in the digit expansion. Then, $\cs{M}_b$ contains infinitely many primes and 
\[
\#\{p\in \cs{M}_b(x) \; : \; p \mbox{ is prime}\}\asymp \frac{\# \cs{M}_b}{\log{x}}.
\]
\end{theorem}
As a final application of our method, we use the $L^{1}$ and $L^{\infty}$ bounds from Maynard \cite{MaynardDigitsOfPrimes,Maynard} to prove a weaker version of Theorem \ref{th: ErdosMauduitSarkozy} that does not rely on equidistribution estimates:
\begin{theorem}\label{th: SquarefreeMissingDigits}
Let $\cs{M}_b$ be as in Theorem \ref{th: MaynardMainResult}, and let $\cs{M}_b^{*}=\{a\in \cs{M}_b \; : \; (a,b)=1\}$. Then, for $b\ge 5$ or $b=4$ with $a_0=0,3$, $\cs{M}_b^{*}$ contains infinitely many squarefree integers, and 
\[
  \sum_{\substack{n\in\cs{M}_b^{*}(x) \\n \text{ is squarefree} }}1\sim \frac{\#\cs{M}_b^{*}(x)}{\zeta(2)}\prod_{p|b}\Big(1-\frac{1}{p^2}\Big)^{-1}.
  \]
  If $b=3$ or $b=4$ with $a_0=1,2$, then $\cs{M}_b^{*}$ contains infinitely many cubefree integers, and 
\[
  \sum_{\substack{n\in\cs{M}_b^{*}(x) \\n \text{ is cubefree} }}1\sim \frac{\#\cs{M}_b^{*}(x)}{\zeta(3)}\prod_{p|b}\Big(1-\frac{1}{p^3}\Big)^{-1}.
  \]
\end{theorem}
\begin{remark}
In base $2$, if the missing digit is $1$, then $\cs{M}_2=\{0\}$, so trivially $\cs{M}_2$ does not contain infinitely many squarefree integers. If the missing digit is $0$, then $\cs{M}_2=\{2^n-1\; : \; n\in\bb{Z}^{+}\}$, so that $\cs{M}_2(x)\asymp \log{x}$. In this case, our methods are not strong enough to deal with $\cs{M}_2$ because by \eqref{eq: LowerBoundL1}, it is not possible to have a bound of the form 
\[
\int_{0}^{1}|F_x(t)|\dt \ll \frac{1}{x^{\f{1}{k}+\de}}
\]
for some $k\in\bb{Z}$. Moreover, we don't expect that an exponential sum type approach would work for this, because 
\[
\Big|\sum_{n\le \sqrt{x}} \me((2^n-1)x)\Big|=\Big|\sum_{n\le \sqrt{x}} \me(2^nx)\Big|;
\]
this essentially shows that the Fourier transform can't differentiate between the sets $\{2^n-1 \; : \; n\in\bb{Z}\}$ and $\{2^n \; : \; n\in\bb{Z}\}$. It is still an open problem to show whether or not there are infinitely many squarefree integers of the form $2^n-1$.
\end{remark}
\section{Proof of the main result}
\subsection{Technical Lemmas}
By the orthogonality relations, note that
\begin{align*}
  \sum_{\substack{n\in\cs{A}(x) \\ d|n}}1={} &   \sum_{0\le a\lt d}\f{1}{d}\sum_{n\in\cs{A}(x)}\me\Big(\f{an}{d}\Big) \\
  ={} & \sum_{0\le a\lt d}\f{1}{d}S_x\Big(\f{a}{d}\Big) \\
  ={} & \f{\#\cs{A}(x)}{d} + \sum_{0\lt a\lt d}\f{1}{d}S_x\Big(\f{a}{d}\Big).
\end{align*}
Therefore, by the triangle inequality and the definition of $F_x(t)$, we have
\begin{equation}\label{eq: Ad(x)-A(x)/d}
  \Big|\sum_{\substack{n\in\cs{A}(x) \\ d|n}}1-\f{1}{d}\#\cs{A}(x)\Big|\le \frac{\#\cs{A}(x)}{d}\sum_{0\lt a\lt d}F_x\Big(\f{a}{d}\Big).
\end{equation}
What this equation tells us is that in order to understand $\cs{A}$ in arithmetic progressions, it suffices to study its Fourier transform. In other words, type I information of a set is interlaced with its Fourier transform. Now, we are interested in counting the number of integers in $\cs{A}$ that are $k^{th}$ powerfree for some $k\ge 2$. We then have
\begin{align*}
  \sum_{\substack{n\in\cs{A}(x) \\ n \text{ is } k^{\text{th}} \text{ powerfree}}}1={} & \sum_{n\in\cs{A}(x)} \sum_{q^k|n}\mu(q) \\
  ={} & \sum_{q\le x^{\f{1}{k}}} \mu(q)\sum_{\substack{n\in\cs{A}(x) \\ q^k|n}}1 \\
  ={} & \#\cs{A}(x)\sum_{q\le x^{\f{1}{k}}} \frac{\mu(q)}{q^k}+\sum_{q\le x^{\f{1}{k}}}\mu(q)\Big(\sum_{\substack{n\in\cs{A}(x) \\ q^k|n}}1-\frac{\#\cs{A}(x)}{q^k}\Big) \\
  ={}& \frac{\#\cs{A}(x)}{\ze(k)}+o\Big(\frac{\#\cs{A}(x)}{x^{\f{1}{k}}}\Big)+O\Big(\#\cs{A}(x) \sum_{q\le x^{\f{1}{k}}} \frac{1}{q^k}\sum_{0\lt a \lt q^k} F_x\Big(\frac{a}{q^k}\Big)\Big),
\end{align*}
where the last equality follows from \eqref{eq: Ad(x)-A(x)/d} with $d=q^k$ and the classical result $\sum_{n\ge 1}\mf{\mu(n)}{n^k}=\mf{1}{\zeta(k)}$. This immediately proves the following result:
\begin{lem}\label{lem: DoubleSumFourierTransform}
  Suppose that $\cs{A}$ is a set of integers such that its normalized Fourier transform $F_x$ satisfies 
  \[
  \sum_{q\le x^{\f{1}{k}}} \frac{1}{q^k}\sum_{0\lt a \lt q^k}F_x\Big(\frac{a}{q^k}\Big)=o(1) 
  \]
  as $x\to\infty$ for some positive integer $k\ge 2$. Then $\cs{A}$ contains infinitely many $k^{th}$ powerfree integers and 
  \[
  \sum_{\substack{n\in\cs{A}(x) \\ n \text{ is } k^{\text{th}} \text{ powerfree}}}1\sim \frac{\#\cs{A}(x)}{\ze(k)} 
  \]
  as $x\to\infty$.
\end{lem}
In order to bound this double sum associated to the Fourier transform, we mainly follow the ideas of Maynard \cite{Maynard} from his work on primes with missing digits. We start with the following easy lemma:
\begin{lem}\label{lem: Large sieve type lemma}
  Let $f\in C^1(\bb{R},\mathbb{C})$ be a periodic function with period $1$. Then, for any integer $k\ge 1$, 
\[
\sum_{0\lt a\lt q^k}\bigg| f\Big(\frac{a}{q^k}\Big)\bigg|\ll q^k\int_{0}^{1}|f(t)|\dt + \int_{0}^{1}|f'(t)|\dt.
\]
\end{lem}
\begin{proof}
  By the Fundamental Theorem of Calculus, 
  \[
  -f\Big(\frac{a}{q^k}\Big)=-f(t)+\int_{\frac{a}{q^k}}^{t}f'(x)\dx,
  \]
  so that, by the triangle inequality, we have
  \begin{equation}\label{eq: EquationToIntegrate}
    \bigg| f\Big(\frac{a}{q^k}\Big)\bigg|\le |f(t)|+\int_{\frac{a}{q^k}}^{t}|f'(x)|\dx.
  \end{equation}
  Now, let $\delta\gt 0$. Then, integrating \eqref{eq: EquationToIntegrate} with respect to $t$ over the interval 
  \[
  I\defeq\Big(\frac{a}{q^k}-\frac{\delta}{2},\frac{a}{q^k}+\frac{\delta}{2}\Big),
  \]
  shows that
  \[
   \delta \bigg| f\Big(\frac{a}{q^k}\Big)\bigg|\le \int_{I}|f(t)|\dt +\int_{I}\int_{\frac{a}{q^k}}^{t}|f'(x)|\dx\dt,
  \]
  and since $x\in\Big(\mfrac{a}{q^k},t\Big)$, then $x\in I$, so that 
  \begin{align*}
    \delta \bigg| f\Big(\frac{a}{q^k}\Big)\bigg|\le{} & \int_{I}|f(t)|\dt +\int_{I}\int_{I}|f'(x)|\dx\dt \\
    ={} & \int_{I}|f(t)|\dt +\delta\int_{I}|f'(x)|\dx.
  \end{align*}
  Dividing by $\delta$ allows us to obtain 
  \begin{equation}\label{eq: BoundFor f(a/q^k)}
    \bigg| f\Big(\frac{a}{q^k}\Big)\bigg|\le \frac{1}{\delta}\int_{I}|f(t)|\dt + \int_{I}|f'(t)| \dt.
  \end{equation}
  Now, let $\delta\defeq q^{-k}$, and observe that with this choice of $\delta$, the intervals $I$ don't overlap for $0\lt a\lt q^k$. Moreover, these $I$ are contained in the interval $(0,1)$, and so summing \eqref{eq: BoundFor f(a/q^k)} over $0\lt a\lt q^k$ shows that 
  \[
  \sum_{0\lt a \lt q^k}\bigg| f\Big(\frac{a}{q^k}\Big)\bigg|\le q^k\int_{0}^1|f(t)|\dt + \int_{0}^1|f'(t)| \dt. \qedhere
  \]
\end{proof}
\subsection{Proof of Theorem \ref{th: Main Theorem}}
For convenience to the reader we restate the theorem here: 
\main*
\begin{proof}
  By Lemma \ref{lem: DoubleSumFourierTransform}, it suffices to show that 
   \[
  \sum_{q\le x^{\f{1}{k}}} \frac{1}{q^k}\sum_{0\lt a \lt q^k}F_x\Big(\frac{a}{q^k}\Big)=o(1)
  \]
  as $x\to\infty$. Moreover, without loss of generality, it suffices to prove the result for $x=b^m$ as $m\to\infty$. The main idea is to write  
  \[
  \sum_{q\le x^{\f{1}{k}}} \frac{1}{q^k}\sum_{0\lt a \lt q^k}F_x\Big(\frac{a}{q^k}\Big)=S_1+S_2,
  \]
  where 
  \[
  S_1\defeq \sum_{q\le(\log{x})^{\f{B}{k}}} \frac{1}{q^k}\sum_{0\lt a \lt q^k}F_x\Big(\frac{a}{q^k}\Big),
  \]
  and 
  \[
  S_2 \defeq \sum_{(\log{x})^{\f{B}{k}}\lt q\le x^{\f{1}{k}}} \frac{1}{q^k}\sum_{0\lt a \lt q^k}F_x\Big(\frac{a}{q^k}\Big).
  \]
  In order to bound $S_1$, we employ the $L^{\infty}$ estimate, and for $S_2$, using a large sieve type estimate, we show that the $L^{1}$ estimate is sufficient to obtain the required bound. With these ideas in mind, we begin by bounding $S_1$: using \eqref{eq: L^infinity estimate}, we obtain for some absolute constant $c$ 
  \[
  S_1\ll \sum_{q\le (\log{x})^{\f{B}{k}}} \exp\Big(-\frac{c}{k}\frac{\log{x}}{\log{q}}\Big)\ll  (\log{x})^{\f{B}{k}}\exp\Big(-\frac{c}{B}\frac{\log{x}}{\log{\log{x}}}\Big)=o(1).
  \]
  To bound $S_2$, we employ Lemma \ref{lem: Large sieve type lemma} to see that for any $u$, 
  \[
  \frac{1}{q^k}\sum_{0\lt a\lt q^k}F_u\Big(\frac{a}{q^k}\Big) \ll \int_{0}^{1}F_u(t)\dt + \f{1}{q^k}\int_{0}^{1}|F_u'(t)|\dt.
  \]
  This together with \eqref{eq: L^1 bounds} shows that 
  \[
   \frac{1}{q^k}\sum_{0\lt a \lt q^k}F_u\Big(\frac{a}{q^k}\Big) \ll \frac{1}{u^{\f{1}{k}+\de}}+ \f{u^{1-\f{1}{k}-\de}}{q^k}.
  \]
  We now choose $u=b^{l}$ maximally subject to $u\le x$ and $u\le q^k$, so that by \eqref{eq: Fourier transform decreasing in x}, we have 
  \[
   \frac{1}{q^k}\sum_{0\lt a \lt q^k}F_x\Big(\frac{a}{q^k}\Big) \ll \frac{1}{x^{\f{1}{k}+\de}}+ \f{1}{q^{1+k\de}}.
  \]
  Finally, summing over $(\log{x})^{\f{B}{k}}\lt q\le x^{\f{1}{k}}$ shows that 
  \begin{align*}
    S_2\ll{} & \sum_{(\log{x})^{\f{B}{k}}\lt q\le x^{\f{1}{k}}} \Big(\frac{1}{x^{\f{1}{k}+\de}}+ \f{1}{q^{1+k\de}}\Big)\\
    \ll{} & \frac{1}{x^{\de}}+\int_{(\log{x})^{\f{B}{k}}}^{x^{\f1k}}\frac{\dt}{t^{1+k\de}} \\
    ={}&\frac{1}{x^{\de}}+\frac{1}{k\de}\Big(\f{1}{(\log{x})^{B\de}}-\f{1}{x^{\de}} \Big) \\ 
    ={}& o(1). \qedhere
  \end{align*}
\end{proof}
\begin{exmp}
To illustrate Theorem \ref{th: Main Theorem}, let us consider the simplest example when $\cs{A}=\bb{Z}^{+}\cup \{0\}$. In this case, 
\[
F_x(t)=\frac{1}{x}\Big| \sum_{n=0}^{x}\me(nt) \Big| = \frac{1}{x}\Big| \frac{\me(xt)-1}{\me(t)-1}\Big|=\frac{1}{x}\Big| \frac{\sin(\pi xt)}{\sin(\pi t)}\Big| 
\]
is essentially the normalized \emph{Dirichlet kernel}. We now show that $F_x(t)$ satisfies the hypotheses of Theorem \ref{th: Main Theorem}:
\begin{enumerate}[{\rm (i)}]
    \item  Let $a,d\in \bb{Z}^{+}$ with $a\lt d$ and $d\ge 2$. Let $\norm{\cdot}$ denote the distance to the nearest integer function. Observe that for any $u \in\bb{R}$, since $|\sin(\pi u)|\ge 2\norm{u}$, then  
        \[
        F_x\Big(\frac{a}{d}\Big)\le \frac{1}{2x\norm{\f{a}{d}}} \le \frac{d}{x}.
        \]
        Now, since $\log{x}\ll x^{\ve}$ for any $\ve\gt 0$, it is clear that if we assume that $d\le \log{x}$, we have 
  \[
  F_x\Big(\frac{a}{d}\Big)\le \frac{\log{x}}{x} \ll \frac{x^{1-\tfrac{1}{2\log{2}}}}{x}\le \f{x^{1-\tfrac{1}{2\log{d}}}}{x}=\exp{\Big(-\frac{\log{x}}{2\log{d}}\Big)}. 
  \]
    \item By symmetry, observe that 
    \begin{equation}\label{eq: Symmetry}
      \int_{0}^{1}F_x(t)\dt =2\int_{0}^{\f12}F_x(t)\dt. 
    \end{equation}
    Now, again using the inequality $|\sin(\pi x)|\ge 2\norm{x}$, it is clear that 
    \[
    F_x(t)\le \min\Big\{1,\frac{1}{2x\norm{t}}\Big\}. 
    \]
    Hence,
    \begin{align*}
      \int_{0}^{\f12}F_x(t)\dt={} & \int_{0}^{\f{1}{x}}F_x(t)\dt+\int_{\f{1}{x}}^{\f12}F_x(t)\dt \\
      \le{} &  \int_{0}^{\f{1}{x}}\dt+\int_{\f{1}{x}}^{\f12}\frac{1}{2xt}\dt \\
      \ll{} & \frac{\log{x}}{x}.
    \end{align*}
    This together with \eqref{eq: Symmetry} shows that 
  \begin{equation*}
  \int_{0}^{1}F_x(t)\dt \ll \frac{1}{x^{1-\ve}}
  \end{equation*}
  for any $\ve\gt 0$. A completely analogous argument allows us to obtain the same bound for $\norm{F_x'(t)}_1$, but multiplied by $x$, showing that $F_x(t)$ satisfies \eqref{eq: L^1 bounds} for any $k \ge 2$.
  \item Observe that for any $m\in\bb{Z}^{+}$, we have 
  \[
  |\sin(2^{m+1}\pi t)|=2|\sin(2^{m}\pi t)\cos(2^{m}\pi t)|\le 2|\sin(2^{m}\pi t)|.
  \]
  Therefore,
  \[
  \frac{1}{2^{m+1}}\Big| \frac{\sin(2^{m+1}\pi t)}{\sin(\pi t)} \Big|\le \frac{1}{2^{m}}\Big| \frac{\sin(2^{m}\pi t)}{\sin(\pi t)} \Big|.
  \]
  Applying this inequality inductively, shows that for any $x,y$ powers of $2$ with $x\le y$, we have 
  \begin{equation*}
  F_y(t)\ll F_x(t).
  \end{equation*}
  \end{enumerate}
  This shows that $F_x(t)$ satisfies all the assumptions of Theorem \ref{th: Main Theorem} for any $k\ge 2$, and we conclude that the set of positive integers contains infinitely many $k^{th}$ powerfree integers, and we recover the classic result 
  \[
  \sum_{\substack{n\le x \\ n \text{ is } k^{\text{th}} \text{ powerfree}}}1\sim \frac{x}{\ze(k)}.
  \]
\end{exmp}
\section{Cubefree palindromes}
Let $\cs{P}_b$ denote the set of non negative palindromes in base $b\ge 2$. It will be convenient to consider the set of palindromes with exactly $l$ digits. For $l\ge 2$, we define
\[
\cs{B}_l\defeq \{n\in\cs{P}_b \; : \; b^{l-1}\le n \lt b^l\}.
\]
For convenience, we define $\cs{B}_1=\{0,\ldots,b-1\}$. We also define a Fourier transform on $\cs{B}_l$:
\[
f_l(t)\defeq \sum_{n\in\cs{B}_l}\me(nt).
\]
First, let's see how we can explicitly compute $f_{2l}(t)$: note that $n\in \mc{B}_{2l}$ if and only if $n=n_0+n_1b+\ldots + n_{2l-1}b^{2l-1}$ for some $n_i\in \{0,\ldots, b-1\}$ with $n_i=n_{2l-i}$ for $i=0,\ldots, l-1$ and $n_0\ge 1$. Then,
\begin{align*}
  f_{2l}(t)={} & \sum_{\substack{n_0,\ldots, n_{2l-1}\in\{0,\ldots,b-1\} \\ n_0\ge 1 \\ n_i=n_{2l-1-i}, \; 0\le i\le 2l-1 }} \me\Big( \sum_{i=0}^{2l-1}n_i b^i\Big)\\
  ={} & \sum_{n_0=1}^{b-1}\me((1+b^{2l-1})n_ot)\prod_{i=1}^{l-1}\sum_{n_i=0}^{b-1}\me((b^i+b^{2l-1-i})n_it) \\
  ={} & (f_1((1+b^{2l-1})t)-1)\prod_{i=1}^{l-1}f_1((b^i+b^{2l-1-i})t).
\end{align*}
Similarly,
\begin{equation}\label{eq: ProductFormula f odd}
f_{2l+1}(t)=(f_1((1+b^{2l})t)-1)f_1(b^lt)\prod_{i=1}^{l-1}f_1((b^i+b^{2l-i})t).
\end{equation}
This shows that the expressions for $f_j(t)$ are very explicit, because 
\[
f_1(t)=\sum_{0\le n\lt b} \me(nt)=\frac{\me(bt)-1}{\me(t)-1}.
\]
Now, note that since $f_1(0)=b$, then the product formula $\eqref{eq: ProductFormula f odd}$ shows that 
\[
\cs{B}_{2l+1}=f_{2l+1}(0)=(b-1)b^l.
\]
Since the expression for $f_l(t)$ differs according to the parity of $l$, for convenience we will work only with palindromes with an odd number of digits, so we let 
\[
\cs{A}\defeq \{n\in \cs{P}_b \; : \; n \mbox{ has an odd number of digits}\}=\bigcup_{l=0}^{\infty}\cs{B}_{2l+1}.
\]
Then, if $S_x(t)=\sum_{n\in\cs{A}(x)}\me(nt)$ denotes the Fourier transform of $\cs{A}(x)$, and if $x=b^{2L+1}$, we then have
\[
S_x(t)=\sum_{l=0}^{L}f_{2l+1}(t).
\]
Following \cite{TuxanidyPanario}, we define 
\[
\Phi_l(t)\defeq \prod_{i=1}^{l-1}f_1((b^i+b^{2l-i})t),
\]
and we observe that 
\begin{equation}\label{eq: |S(t)| less than a sum over Phi}
  |S_x(t)|\le \sum_{l=0}^{L}|f_{2l+1}(t)| \le b^2\sum_{l=0}^{L}|\Phi_l(t)|,
\end{equation}
where we define $\Phi_0(t)=\Phi_1(t)=1$. Hence, in order to study the Fourier transform of the palindromes, it essentially suffices to study $\Phi_l(t)$. Moreover, since it is convenient to work with a normalized version, we define
\[
\tilde{\Phi}_l(t)\defeq \frac{\Phi_l(t)}{\#\cs{B}_{2l+1}}=\frac{\Phi_l(t)}{(b-1)b^l}.
\]
We will see that $\tilde{\Phi}_l(t)$ will play the role of $F_x(t)$ from the previous sections, because it will satisfy very similar properties to the ones in the assumptions of Theorem \ref{th: Main Theorem}. It will also later become clear why we work with $\Phi_l(t)$ instead of $f_{2l+1}(t)$. We start with the $L^{\infty}$ bound (the analogous of \eqref{eq: L^infinity estimate}) first proved by Col \cite{Col}, and then improved by Tuxanidy and Panario \cite{TuxanidyPanario}:
\begin{lem}[\cite{TuxanidyPanario}, Proposition 6.2]\label{lem: L Infinity bound}
Let $a,d,m,l$ be integers with $d\ge 2$, $(d,a(b^3-b))=1$. Then, 
\[
\Big|\tilde{\Phi}_l\Big(\frac{a}{d}+\frac{m}{b^3-b}\Big)\Big| \ll_{b} \exp\Big(-c_b \frac{l}{\log{d}}\Big),
\]
where $c_b\gt 0$ is some absolute constant depending only on $b$.
\end{lem}
This shows that $\tilde{\Phi}_l(a/d)$ satisfies a slight extension of \eqref{eq: L^infinity estimate}, but with the more restrictive condition that $d$ is relativity prime to both $a$ and $b^3-b$. This leads us to consider the more restrictive set (just as in \cite{TuxanidyPanario} and \cite{ShashiDaniel}) 
 \[
  \cs{P}_b^{*}\defeq \{n\in\cs{P}_b \; : \; (n,b^3-b)=1\}.
 \]
Now we prove the analogue of \eqref{eq: Fourier transform decreasing in x}:
\begin{lem}\label{lem: Decreasingness of Phi}
  Let $l_1,l_2,d$ be positive integers such that $l_2\le l_1$ and $(d,b^3-b)=1$. Then, 
  \begin{equation}\label{eq: Decreasing of Phi on average}
  \sum_{\substack{0\lt a \lt d \\ (a,d)=1}}\max_{m\in\bb{Z}}\Big|\tilde{\Phi}_{l_1}\Big(\f{a}{d}+\frac{m}{b^3-b}\Big)\Big|\le \sum_{\substack{0\lt a \lt d \\ (a,d)=1}}\max_{m\in\bb{Z}}\Big|\tilde{\Phi}_{l_2}\Big(\f{a}{d}+\frac{m}{b^3-b}\Big)\Big|
  \end{equation}
\end{lem}
\begin{proof}
  From the definition of $\Phi_l(t)$ we observe that 
  \[
  \Phi_l(t)=f_1((b+b^{2l-1})t)\Phi_{l-1}(bt). 
  \]
  Hence, using the trivial bound $|f(\theta)|\le b$ for any $\theta$, we have 
  \[
  |\Phi_l(t)|\le b|\Phi_{l-1}(bt)|.
  \]
  This shows that 
  \begin{equation}\label{eq: Phi decreasing in l}
   |\tilde{\Phi}_l(t)|=\frac{|\Phi_l(t)|}{(b-1)b^l}\le \frac{|\Phi_l(bt)|}{(b-1)b^{l-1}}=|\tilde{\Phi}_{l-1}(bt)|.
  \end{equation}
  Using this equation inductively we see that if $l_2\le l_1$, then 
  \[
   |\tilde{\Phi}_{l_1}(t)|\le  |\tilde{\Phi}_{l_2}(b^{l_1-l_2}t)|.
  \]
   Now, letting $t=\f{a}{d}+\frac{m}{b^3-b}$, and summing over $0\lt a\lt d$ with $(a,d)=1$ completes the proof upon observing that since $(d,b^3-b)=1$ and $(a,d)=1$, then multiplication by a power of $b$ keeps the sum on the right hand side of \eqref{eq: Decreasing of Phi on average} invariant.
\end{proof}
\begin{remark}
The main reason we work with $\Phi_l(t)$ instead of $f_{2l+1}(t)$ is because the later does not satisfy a condition as nice as \eqref{eq: Decreasing of Phi on average}. In fact, a recursive formula for $f_{2l+1}(t)$ involves all the functions $f_{2j+1}(t)$ for $j=0,\ldots,l-1$: 
\[
f_{2l+1}(t)=(f_1((1+b^{2l})t)-1)\sum_{j=0}^{l-1}f_{2j+1}(b^{l-j}t).
\]
\end{remark}
Now we are interested in $L^{1}$ bounds for $\Phi_l$ and $\Phi_l'$. Note that by Parseval's identity, 
\[
\int_{0}^{1}|f_{2l+1}(t)|^2\dt = \sum_{n\in\cs{B}_{2l+1}} 1^2= (b-1)b^l \quad\mbox{and}\quad \int_{0}^{1}|f_{2l+1}'(t)|^2\dt\ll b^{5l}
\]
so that by the Cauchy-Schwartz inequality we have 
\[
\int_{0}^{1}|f_{2l+1}(t)|\dt\ll b^{l/2}\quad\mbox{and}\quad \int_{0}^{1}|f_{2l+1}'(t)|\dt\ll b^{5l/2}.
\]
We expect that $\Phi_l(t)$ satisfies the same bounds as above. However, this is not as straightforward to show, because when expanding the product as 
\[
|\Phi_l(t)|^2=|\Phi_l(t) \Phi_l(-t)|=b^{l-1}+\sum_{1\le n\le N} a_n \me (\al_n t)
\]
for some $N\in\bb{Z}^{+}$ and some sequences $a_n, \al_n$, it is not trivial to show that $\al_n\neq 0$. Moreover, even if this holds, we would only get a ``trivial'' bound for $\norm{\Phi_l}_1$. Therefore, we use a different approach, using similar ideas to the ones from Maynard \cite{MaynardDigitsOfPrimes} to show that for sufficiently large base $b$, we can do better. 
\begin{lem}\label{lem: L1 Bound for Phi}
Let 
\[
\alpha_b\defeq \frac{\log{b}-\log(4+\log(\f{b}{2}-1))}{2\log{b}}.
\]
Then, for $x=b^{2l}$, we have 
\[
\int_{0}^{1}|\tilde{\Phi}_l(t)|\dt \ll \frac{1}{x^{\alpha_b}} \quad \mbox{and}\quad \int_{0}^{1}|\tilde{\Phi}_l'(t)|\dt \ll l x^{1-\alpha_b}.
\]
In particular, for every $b\ge 1100$, there is some $\de_b\gt 0$ such that 
\[
\int_{0}^{1}|\tilde{\Phi}_l(t)|\dt \ll \frac{1}{x^{\f13+\de_b}} \quad \mbox{and}\quad \int_{0}^{1}|\tilde{\Phi}_l'(t)|\dt \ll x^{\f23-\de_b}.
\]
\end{lem}
\begin{proof}
  We begin by noting that 
  \begin{align}
    \int_{0}^{1}|\Phi_l(t)|\dt={} & \sum_{0\le a\lt b^{2l}}\int_{\f{a}{b^{2l}}}^{\f{a}{b^{2l}}+\f{1}{b^{2l}}}|\Phi_l(t)|\dt \nonumber\\ 
    ={} & \sum_{0\le a\lt b^{2l}}\int_{0}^{\f{1}{b^{2l}}} \Big|\Phi_l\Big(\frac{a}{b^{2l}}+\eta\Big)\Big|\d\eta \nonumber\\
    \le{} & \frac{1}{b^{2l}}\sum_{0\le a\lt b^{2l}}\sup_{\eta\in [0,\f{1}{b^{2l}})}\Big|\Phi_l\Big(\frac{a}{b^{2l}}+\eta\Big)\Big| \nonumber\\
    ={}& \frac{1}{b^{2l}}\sum_{a_1,a_2,\ldots,a_{2l}\in\{0,\ldots,b-1\}}\sup_{\eta\in [0,\f{1}{b^{2l}})}\Big|\Phi_l\Big(\f{a_1}{b}+\f{a_2}{b^2}+\cdots+\f{a_{2l}}{b^{2l}}+\eta\Big)\Big|. \label{eq: L1Bound via sum of sup of digits}
  \end{align}
 Let $\norm{\cdot}$ denote the distance to the nearest integer function. Note that for any $y \in\bb{R}$, since $|\sin(\pi y)|\ge 2\norm{y}$, then 
  \[
  |f_1(y)|=\Big|\frac{\me(by)-1}{\me(y)-1}\Big| = \Big|\frac{\sin(\pi by)}{\sin(\pi y)}\Big|\le\min\Big\{b, \frac{1}{2\norm{y}}\Big\}.
  \]
  This together with the definition of $\Phi_l(t)$ shows that 
  \begin{align}
    |\Phi_l(t)|={} & \prod_{i=1}^{l-1}\Big|\frac{\sin(\pi(b^{i+1}+b^{2l+1-i})t)}{\sin(\pi(b^{i}+b^{2l-i})t)}\Big| \nonumber\\
    \le {}&  \prod_{i=1}^{l-1}\min\Big\{b, \frac{1}{2\norm{(b^{i}+b^{2l-i})t}}\Big\}\nonumber \\
    ={} & b^{l-1}\prod_{i=1}^{l-1} \min\Big\{1, \frac{1}{2b\norm{(b^{i}+b^{2l-i})t}}\Big\}.\label{eq: Bound Phi via min}
  \end{align}
  Now, for $t\in[0,1)$, by writing 
  \[
  t=\f{t_1}{b}+\f{t_2}{b^2}+\cdots +\f{t_i}{b^i}+ \cdots+\f{t_{2l-i}}{b^{2l-i}}+\cdots,
  \]
  we see that 
  \[
  b^i t=\mbox{Integer} + \f{t_{i+1}}{b}+u_i \quad \mbox{and}\quad b^{2l-i} t=\mbox{Integer} + \f{t_{2l+1-i}}{b}+v_i
  \]
  for some $u_i,v_i \in [0,1/b)$, showing that 
  \[
  \norm{(b^i+b^{2l-i})t}=\norm{\f{t_{i+1}}{b}+\f{t_{2l+1-i}}{b}+w_i}
  \]
  for some $w_i\in [0,2/b)$. This together with \eqref{eq: L1Bound via sum of sup of digits} and \eqref{eq: Bound Phi via min} shows that 
  \[
  \int_{0}^{1}|\Phi_l(t)|\dt \ll \f{1}{b^l} \sum_{a_1,\ldots, a_{2l-1}\in\{0,\ldots,b-1\}}\prod_{i=1}^{l-1}\sup_{\beta\in[0,\f{2}{b})}\min\Big\{1,\f{1}{2b\norm{\f{a_i}{b}+\f{a_{2l-i}}{b}+\beta}}\Big\}.
  \]
  Given digits $\theta_1,\theta_2\in \{0,\ldots, b-1\}$, we let 
  \[
  G(\theta_1,\theta_2)\defeq \sup_{\beta\in[0,\f{2}{b})}\min\Big\{1,\f{1}{2b\norm{\f{\theta_1}{b}+\f{\theta_2}{b}+\beta}}\Big\}.
  \]
  Then, 
  \begin{align}
    \int_{0}^{1}|\Phi_l(t)|\dt\ll{} & \f{1}{b^l} \sum_{a_1,\ldots, a_{2l-1}\in\{0,\ldots,b-1\}}G(a_1,a_{2l-1})G(a_2,a_{2l-2})\cdots G(a_{l-1},a_{l+1}) \nonumber \\
    \ll{} & \f{1}{b^l} \Big(\sum_{\theta_1,\theta_2\in\{0,\ldots,b-1\}}G(\theta_1,\theta_2) \Big)^l.\label{eq: L1 bound in terms of G}
  \end{align}
  In order to understand $G(\theta_1,\theta_2)$ we consider the $b\times b$ matrix 
  \[
  M_b\defeq [G(i,j)]_{i,j\in\{0,\ldots,b-1\}},
  \]
  and we observe that $M_b$ satisfies the following properties:
  \begin{enumerate}[{\rm (i)}]
    \item $M_b$ is symmetric because $G(\theta_1,\theta_2)=G(\theta_2,\theta_1)$.
    \item $M_b$ is a Hankel matrix (i.e., a matrix in which each ascending skew-diagonal from left to right is constant) because 
    \[
        G(\theta_1,\theta_2)=G(\theta_1+1,\theta_2-1).
   \]
    \item For any $\theta_2\in\{0,\ldots,b-2\}$, we have $G(0,\theta_2)=G(0,b-2-\theta_2)$.
  \end{enumerate}
  For example, in base $b=10$, the matrix looks like this:
  \[M_{10}=
  \begin{bmatrix}
    1 & \f12 & \f14 & \f16 & \f18 & \f16 & \f14 & \f12 & 1 & 1 \\[5pt]
    \f12 & \f14 & \f16 & \f18 & \f16 & \f14 & \f12 & 1 & 1 & 1 \\[5pt]
    \f14 & \f16 & \f18 & \f16 & \f14 & \f12 & 1 & 1 & 1 & \f12 \\[5pt]
    \f16 & \f18 & \f16 & \f14 & \f12 & 1 & 1 & 1 & \f12 & \f14 \\[5pt]
    \f18 & \f16 & \f14 & \f12 & 1 & 1 & 1 & \f12 & \f14 & \f16 \\[5pt]
    \f16 & \f14 & \f12 & 1 & 1 & 1 & \f12 & \f14 & \f16 & \f18 \\[5pt]
    \f14 & \f12 & 1 & 1 & 1 & \f12 & \f14 & \f16 & \f18 & \f16 \\[5pt]
    \f12 & 1 & 1 & 1 & \f12 & \f14 & \f16 & \f18 & \f16 & \f14 \\[5pt]
    1 & 1 & 1 & \f12 & \f14 & \f16 & \f18 & \f16 & \f14 & \f12 \\[5pt]
    1 & 1 & \f12 & \f14 & \f16 & \f18 & \f16 & \f14 & \f12 & 1 
  \end{bmatrix}
  \]
  From the above properties, it is clear that 
  \begin{equation}\label{eq: DoubleSum of G}
  \sum_{\theta_1,\theta_2\in\{0,\ldots,b-1\}}G(\theta_1,\theta_2)=b\sum_{0\le \theta_2\le b-1}G(0,\theta_2).
  \end{equation}
  Firsts let's assume that $b$ is even. Then, for $\theta_2\le \f{b}{2}-2$, $\f{\theta_2}{b}+\be\le \f{1}{2}$, so that
  \[
  \norm{\f{\theta_2}{b}+\be}=\f{\theta_2}{b}+\be\ge\f{\theta_2}{b},
  \]
  showing that for $\theta_2\neq 0$, 
  \[
  G(0,\theta_2)\le \frac{1}{2\theta_2}.
  \]
  Now, it is easy to see that $G(0,\f{b}{2}-1)=\f{1}{b-2}$, and also $G(0,0)=G(0,b-1)=G(0,b-2)=1$. Combining this with the identity $G(0,\theta_2)=G(0,b-2-\theta_2)$ gives us the estimate
  \[
  \sum_{0\le \theta_2\le b-1}G(0,\theta_2)\le 3+\sum_{1\le \theta_2\le \f{b}{2}-1}\f{1}{\theta_2}\le 4+\log\Big(\f{b}{2}-1\Big).
  \]
  A similar argument shows that the same bound holds when $b$ is odd. Combining this with \eqref{eq: L1 bound in terms of G} and \eqref{eq: DoubleSum of G} gives us 
  \[
  \int_{0}^{1}|\Phi_l(t)|\dt \ll \Big(4+\log\Big(\f{b}{2}-1\Big)\Big)^l.
  \]
  After normalizing, we obtain 
  \[
  \int_{0}^{1}|\tilde{\Phi}_l(t)|\dt \ll \Big(\frac{4+\log\Big(\f{b}{2}-1\Big)}{b}\Big)^l=\frac{1}{x^{\al_b}},
  \]
  upon recalling that $x=b^{2l}$ and 
  \[
  \alpha_b= \frac{\log{b}-\log(4+\log(\f{b}{2}-1))}{2\log{b}}.
  \]
  The bounds for $\Phi_l'(t)$ are completely analogous because 
  \[
  \Phi_l'(t)=\sum_{j=1}^{l-1}f_1'((b^{j}+b^{2l-j})t)(b^{j}+b^{2l-j})\prod_{\substack{i=1 \\ i\neq j}}^{l-1}f_1((b^{i}+b^{2l-i})t),
  \]
  so that 
  \[
  |\Phi_l'(t)|\ll b^{2l}\sum_{j=1}^{l-1}\prod_{\substack{i=1 \\ i\neq j}}^{l-1}f_1((b^{i}+b^{2l-i})t).
  \]
  Hence, by a completely analogous argument as before, we see that we have the same bound as for $\norm{\Phi_l}_1$, but multiplied by $lb^{2l}=lx$. Finally, a numerical computation shows that 
  \[
  \al_{1100}\approx 0.333445\ldots \gt \f13.
  \]
  This completes the proof.
\end{proof}
We are now ready to combine our $L^{\infty}$ and $L^{1}$ estimates together with the ``decreasing'' property of $\Phi_l(t)$ in order to prove Theorem \ref{th: Cubefree}. We restate the theorem here for convenience to the reader: 
\cubefree*
\begin{proof}
  We follow an analogous procedure as in our proof of Theorem \ref{th: Main Theorem}, and we start by writing
  \begin{align*}
  \sum_{\substack{n\in\cs{P}_b^{*}(x) \\ n \text{ is cubefree}}}1={} & \sum_{n\in\cs{P}_b^{*}(x)} \sum_{q^3|n}\mu(q) \\
  ={} & \sum_{\substack{q\le x^{\f{1}{3}}\\ (q,b^3-b)=1}} \mu(q)\sum_{\substack{n\in\cs{P}_b^{*}(x) \\ q^3|n}}1 \\
  ={} &\sum_{\substack{q\le x^{\f{1}{3}}\\ (q,b^3-b)=1}} \mu(q)\frac{\#\cs{P}_b^{*}(x)}{q^3}+\sum_{\substack{q\le x^{\f{1}{3}}\\ (q,b^3-b)=1}}\mu(q)\Big( \sum_{\substack{n\in\cs{P}_b^{*}(x) \\ q^3|n}}1-\frac{\#\cs{P}_b^{*}(x)}{q^3}\Big)\\
  ={}& \frac{\cs{P}_b^{*}(x)}{\zeta(3)}\prod_{p|b^3-b}\Big(1-\frac{1}{p^3}\Big)^{-1}+o\Big(\frac{\#\cs{P}_b^{*}(x)}{x^{\f13}} \Big)+O(E),
\end{align*}
where 
\[
E\defeq \sum_{\substack{q\le x^{\f{1}{3}}\\ (q,b^3-b)=1}}\Big( \sum_{\substack{n\in\cs{P}_b^{*}(x) \\ q^3|n}}1-\frac{\#\cs{P}_b^{*}(x)}{q^3}\Big).
\]
Our objective is to show that $E=o(\#\cs{P}_b^{*}(x))$. Since $\#\cs{P}_b^{*}(x)\asymp \sqrt{x}$ (\cite{TuxanidyPanario}, Lemma 9.1), it suffices to show that $E=o(\sqrt{x})$. In order to do this, we start by relating $\#\cs{P}_b^{*}(x)$ with $\cs{A}(x)$ as follows: first observe that if $n\in\cs{P}_b$ has an even number of digits, then 
\[
n=n_0+n_1b+\cdots + n_{2l-1}b^{2l-1}=n_0(1+b^{2l-1}) +n_1(b+b^{2l-2})+\cdots + n_{l-1}(b^{l-1}+b^{l}),
\]
so that $(b+1) | n$. Therefore, upon recalling that $\cs{A}$ is the set of palindromes with an odd number of digits, we have 
\[
\#\cs{P}_b^{*}(x)=\{n\in\cs{A}(x) \; :\; (n,b^3-b)=1\}.
\]
This together with M\"obius inversion 
\begin{equation}\label{eq: MobiusInversion}
\sum_{d|(n,b^3-b)}\mu(d)=\begin{cases}
                           1, & \mbox{if } (n,b^3-b)=1 \\
                           0, & \mbox{otherwise}
                         \end{cases}
\end{equation}
shows that 
\begin{equation}\label{eq: Pb* in terms of A}
\sum_{\substack{n\in\cs{P}_b^{*}(x) \\ q^3|n}}1-\frac{\#\cs{P}_b^{*}(x)}{q^3}=\sum_{d|b^3-b}\mu(d) \sum_{\substack{n\in\cs{A}(x)\\ d|n}}\Big(\1_{q^3|n}-\f{1}{q^3}\Big).
\end{equation}
By the orthogonality relations, 
\begin{equation}\label{eq: OrthogonalityRelations n divisible by d}
\1_{d|n}=\f{1}{d}\sum_{0\le m\lt d} \me\Big(\f{mn}{d}\Big), 
\end{equation}
and 
\begin{equation}\label{eq: Orthogonality Relations n divisible by q3}
\1_{q^3|n}-\f{1}{q^3}=\f{1}{q^3}\sum_{0\lt a\lt q^3}\me\Big(\f{an}{q^3}\Big).
\end{equation}
Hence, after plugging \eqref{eq: OrthogonalityRelations n divisible by d} and \eqref{eq: Orthogonality Relations n divisible by q3} into \eqref{eq: Pb* in terms of A}, we see that
\begin{align*}
  \bigg|\sum_{\substack{n\in\cs{P}_b^{*}(x) \\ q^3|n}}1-\frac{\#\cs{P}_b^{*}(x)}{q^3}\bigg| \le {}& \sum_{d|b^3-b} \f{1}{d}\sum_{0\le m\lt d}\f{1}{q^3}\sum_{0\lt a\lt q^3}\bigg|\sum_{n\in\cs{A}(x)}\me\Big(\Big(\f{a}{q^3}+\f{m}{d}\Big)n\Big)\bigg| \\
  ={} & \sum_{d|b^3-b} \f{1}{d}\sum_{0\le m\lt d}\f{1}{q^3}\sum_{0\lt a\lt q^3}\Big|S_x\Big(\f{a}{q^3}+\f{m}{d}\Big)\Big| \\
  \ll{}& \f{1}{q^3}\sum_{0\lt a\lt q^3}\max_{m\in\bb{Z}}\Big|S_x\Big(\f{a}{q^3}+\f{m}{b^3-b}\Big)\Big|,
\end{align*}
where $S_x(t)=\sum_{n\in\cs{A}(x)}\me(nt)$ is the Fourier transform of $\cs{A}$. Now, without loss of generality, we may assume that $x=b^{2L+1}$, so that by \eqref{eq: |S(t)| less than a sum over Phi}, 
\[
|S_x(t)|\ll \sum_{l=0}^{L}|\Phi_l(t)|.
\]
Therefore, 
\begin{align*}
  E \ll{}&  \sum_{\substack{q\le x^{\f{1}{3}}\\ (q,b^3-b)=1}}\sum_{l=0}^{L}\f{1}{q^3}\sum_{0\lt a\lt q^3} \max_{m\in\bb{Z}}\Big|\Phi_l\Big(\f{a}{q^3}+\f{m}{b^3-b}\Big)\Big| \\
  \ll{} & \log{x}\sum_{\substack{q\le x^{\f{1}{3}}\\ (q,b^3-b)=1}}\sum_{l=0}^{L}\f{1}{q^3}\sum_{\substack{0\lt a\lt q^3 \\ (a,q^3)=1}} \max_{m\in\bb{Z}}\Big|\Phi_l\Big(\f{a}{q^3}+\f{m}{b^3-b}\Big)\Big|,
\end{align*}
where the last bound follows from writing $\f{a}{q^3}$ in lowest terms. We now write 
\[
\log{x}\sum_{\substack{q\le x^{\f{1}{3}}\\ (q,b^3-b)=1}}\sum_{l=0}^{L}\f{1}{q^3}\sum_{\substack{0\lt a\lt q^3 \\ (a,q^3)=1}} \max_{m\in\bb{Z}}\Big|\Phi_l\Big(\f{a}{q^3}+\f{m}{b^3-b}\Big)\Big|=S_1+S_2, 
\]
where $S_1$ is the sum over $q\le (\log{x})^B$ and $S_2$ is over $(\log{x})^B\lt q \le x^{\f13}$ for some $B$ to be chosen later. Explicitly, we have 
\[
S_1=\log{x}\sum_{\substack{q\le (\log{x})^B\\ (q,b^3-b)=1}}\sum_{l=0}^{L}\f{1}{q^3}\sum_{\substack{0\lt a\lt q^3 \\ (a,q^3)=1}} \max_{m\in\bb{Z}}\Big|\Phi_l\Big(\f{a}{q^3}+\f{m}{b^3-b}\Big)\Big|,
\]
and 
\[
S_2=\log{x}\sum_{\substack{(\log{x})^B\lt q\le x^{\f{1}{3}}\\ (q,b^3-b)=1}}\sum_{l=0}^{L}\f{1}{q^3}\sum_{\substack{0\lt a\lt q^3 \\ (a,q^3)=1}} \max_{m\in\bb{Z}}\Big|\Phi_l\Big(\f{a}{q^3}+\f{m}{b^3-b}\Big)\Big|.
\]
We recall that it is convenient to work with the normalized version of $\Phi_l(t)$:
\[
\tilde{\Phi}_l(t)=\frac{1}{\#\cs{B}_{2l+1}}\Phi_l(t) \asymp \f{1}{b^l} \Phi_l(t).
\]
In order to bound $S_1$, we employ the $L^{\infty}$ bound from Lemma \ref{lem: L Infinity bound}:
\begin{align*}
  S_1\ll{} & \log{x}\sum_{\substack{q\le (\log{x})^B\\ (q,b^3-b)=1}}\sum_{l=0}^{L}\f{1}{q^3}\sum_{\substack{0\lt a\lt q^3 \\ (a,q^3)=1}} b^l \exp\Big(-c_b\frac{l}{\log{q}}\Big) \\
  \ll{} & \log{x}\sum_{q\le (\log{x})^B} L b^{L} \exp\Big(-c_b\frac{L}{\log{q}}\Big)\\
  \ll{} & (\log{x})^{B+2}\exp\Big(-\f{c_b}{B}\frac{\log{x}}{\log{\log{x}}}\Big)\sqrt{x} \\
  ={} & o(\sqrt{x}),
\end{align*}
where the third bound follows from recalling that $x=b^{2L+1}$. For $S_2$, we first choose $l'$ maximally subject to $l'\le l$ and $b^{2l'}\le q^3$ to see that, by Lemma \ref{lem: Decreasingness of Phi}, we have 
\[
\sum_{\substack{0\lt a\lt q^3 \\ (a,q^3)=1}} \max_{m\in\bb{Z}}\Big|\tilde{\Phi}_l\Big(\f{a}{q^3}+\f{m}{b^3-b}\Big)\Big|\le \sum_{\substack{0\lt a\lt q^3 \\ (a,q^3)=1}} \max_{m\in\bb{Z}}\Big|\tilde{\Phi}_{l'}\Big(\f{a}{q^3}+\f{m}{b^3-b}\Big)\Big|.
\]
Now, by Lemma \ref{lem: Large sieve type lemma} and Lemma \ref{lem: L1 Bound for Phi}, for $b\ge 1100$ there is some $\de\defeq \de_b\gt 0$ such that 
\begin{align*}
  \sum_{\substack{0\lt a\lt q^3 \\ (a,q^3)=1}} \max_{m\in\bb{Z}}\Big|\tilde{\Phi}_{l'}\Big(\f{a}{q^3}+\f{m}{b^3-b}\Big)\Big|\ll{} & q^3\int_{0}^{1}|\tilde{\Phi}_{l'}(t)|\dt + \int_{0}^{1}|\tilde{\Phi}_{l'}'(t)|\dt \\
  \ll{} & \frac{q^3}{b^{2l'(\f13+\de)}}+ l'b^{2l'(\f23-\de)}.
\end{align*}
From our choice of $l'$, we observe that 
\[
\frac{q^3}{b^{2l'(\f13+\de)}}+ l'b^{2l'(\f23-\de)}\ll \frac{q^3}{b^{2l(\f13+\de)}}+lq^{2-3\de}.
\]
This shows that 
\begin{align*}
  S_2={} & \log{x}\sum_{\substack{(\log{x})^B\lt q\le x^{\f{1}{3}}\\ (q,b^3-b)=1}}\sum_{l=0}^{L}\f{(b-1)b^l}{q^3}\sum_{\substack{0\lt a\lt q^3 \\ (a,q^3)=1}} \max_{m\in\bb{Z}}\Big|\tilde{\Phi}_l\Big(\f{a}{q^3}+\f{m}{b^3-b}\Big)\Big| \\
  \ll{} & \log{x}\sum_{\substack{(\log{x})^B\lt q\le x^{\f{1}{3}}\\ (q,b^3-b)=1}}\sum_{l=0}^{L}\f{b^l}{q^3} \Big(\frac{q^3}{b^{2l(\f13+\de)}}+lq^{2-3\de} \Big)\\
  \ll{} & \log{x}\sum_{(\log{x})^B\lt q\le x^{\f{1}{3}}}b^L\Big(\f{1}{b^{2L(\f13+\de)}}+\f{L}{q^{1+3\de}} \Big) \\
  \ll{} & \sqrt{x}\log{x}\Big(\f{1}{x^{\de}} + \int_{(\log{x})^B}^{x^{\f13}}\frac{\dt}{t^{1+3\de}}\Big) \\
  ={} & \sqrt{x}\log{x}\Big(\f{1}{x^{\de}} - \frac{1}{3\de x^{\de}}+\frac{1}{3\de(\log{x})^{3\de B}}\Big).
\end{align*}
Therefore, by taking  $B\defeq \f{1}{\de}\gt \f{1}{3\de}$, we see that $S_2=o(\sqrt{x})$, so that 
\[
E\ll S_1+S_2=o(\sqrt{x}).
\]
This completes the proof.
\end{proof}
\section{Squarefree reversible integers}
Recall that
\[
\cs{H}_b= \{n\in\bb{Z}_{\ge 1} \; : \; (n,b^3-b)=1, \; (\mirror{n}_b,b^3-b)=1\}.
\]
Following \cite{DartygeRivatSwaenepoel}, for $x=b^l$, and $\al,\be \in\bb{R}$, we let 
\[
\mathcal{F}_x(\al,\be)=\frac{1}{x}\sum_{0\le n\lt x}\me(\al\mirror{n}_b-\be n)
\]
We state the results we need from \cite{DartygeRivatSwaenepoel}, showing that $\mathcal{F}_x(\al,\be)$ plays the role of the Fourier transform from Theorem \ref{th: Main Theorem}:
\begin{lem}\label{lem: NewCecileResult}
Assume that $x=b^k$ and $y=b^l$ for some $k,l\in\bb{Z}^{+}$ and $b\ge 2$. 
\begin{enumerate}[{\rm (i)}]
  \item (\cite{DartygeRivatSwaenepoel}, Lemma 6.10) Let $d\gt 1$ be an integer such that $(d,b^3-b)=1$. Then, for all $a\lt d$ with $(a,d)=1$, we have an $L^{\infty}$ bound of the form
      \[
      \Big|\mathcal{F}_x\Big(\frac{a}{d},\be \Big)\Big |\ll \exp\Big(-c_b \frac{\log{x}}{\log{d}}\Big)
      \]
  for some positive constant $c_b$ that depends only on $b$. 
  \item (\cite{DartygeRivatSwaenepoel}, Lemma 6.17) We have an $L^{1}$ bound of the form 
  \[
  \int_{0}^{1}|\mathcal{F}_x(\al,\be)|\d\be \ll \frac{1}{x^{1-\eta_b}} \quad\mbox{and}\quad \int_{0}^{1}|\mathcal{F}_x(\al,\be)|\d\be \ll x^{\eta_b}.
  \]
  for some $\eta_b\in (0,0.465)$. In particular, note that $1-\eta_b=\frac{1}{2}+\de$ for some $\de\gt 0$.
  \item (\cite{DartygeRivatSwaenepoel}, Lemma 6.2) If $x\le y$, then $|\mathcal{F}_y(\al,\be)| \le |\mathcal{F}_x(\al y x^{-1},\be)|$ for all $\al,\be\in \bb{R}$.
\end{enumerate}
\end{lem}
The proof of Theorem \ref{th: n and rev(n) squarefree} is completely analogous to the proof of Theorem \ref{th: Cubefree}, so we only sketch the main ideas. We begin by writing
\begin{align*}
  \sum_{\substack{n\in\cs{H}_b(x) \\ n,\mirror{n}_b \text{ are squarefree}}}1={} & \sum_{\substack{n\le x \\ (n\mirror{n}_b, b^3-b)=1}}\sum_{\substack{d_1^2|n \\ d_2^2|\mirror{n}_b}}\mu(d_1)\mu(d_2) \\
  ={}& \sum_{\substack{d_1,d_2\le \sqrt{x}\\ (d_1d_2,b^3-b)=1}}\mu(d_1)\mu(d_2)\Big(\sum_{\substack{n\le x \\ (n\mirror{n}_b,b^3-b)=1 \\d_1^2|n \\ d_2^2|\mirror{n}_b }}1-\frac{\#\cs{H}_b(x)}{d_1^2d_2^2}+\frac{\#\cs{H}_b(x)}{d_1^2d_2^2}\Big) \\
  ={} & \frac{\#\cs{H}_b(x)}{\zeta(2)^2}\prod_{p|b^3-b}\Big(1-\frac{1}{p^2}\Big)^{-2}+o\Big(\frac{\# \cs{H}_b(x)}{x} \Big)+O(E), 
\end{align*}
where
\[
E=\sum_{\substack{n\le x \\ (n\mirror{n}_b,b^3-b)=1 \\d_1^2|n \\ d_2^2|\mirror{n}_b }}1-\frac{1}{d_1^2d_2^2}\sum_{\substack{n\le x \\ (n\mirror{n}_b,b^3-b)=1}}1.
\]
The proof of Theorem \ref{th: n and rev(n) squarefree} will be complete once we show that $E=o(\#\cs{H}_b(x))$, and by the definition of $\cs{H}_b$, it is clear that it suffices to show that $E=o(x)$. By the orthogonality relations \eqref{eq: OrthogonalityRelations n divisible by d} and M\"{o}bius inversion \eqref{eq: MobiusInversion}, observe that 
\[
E=\sum_{\substack{d_1,d_2\le \sqrt{x}\\ (d_1d_2,b^3-b)=1}}\sum_{d|b^3-b} \mu(d) \sum_{n\le x}\frac{1}{d_1^2d_2^2d}\sum_{0\le m\lt d}\sum_{\substack{0\lt a_1\lt d_1\\0\lt a_2\lt d_2 }}\me\Big(\frac{a_1}{d_1^2}n+\frac{a_2}{d_2^2}\mirror{n}_b\Big)\me\Big(\frac{mn\mirror{n}_b}{d}\Big)
\]
Interchanging the order of summation, using the definition of $\mathcal{F}_x(\al,\be)$, and using the condition $d|b^3-b$ gives us 
\[
E\ll x \sum_{d_1,d_2\le \sqrt{x}} \frac{1}{d_1^2d_2^2}\sum_{\substack{0\lt a_1\lt d_1\\0\lt a_2\lt d_2 }}\max_{m\in\bb{Z}}\Big|\mathcal{F}_x\Big(\frac{a_2}{d_2^2},-\frac{a_1}{d_1^2}+\frac{m}{b^3-b}\Big) \Big|.
\]
Finally, we split the sum over $d_2\le \sqrt{x}$ into a sum over $d_2\le (\log{x})^{B}$ for some $B$, and a sum over $(\log{x})^{B}\lt d_2\le \sqrt{x}$, and then we use Lemma \ref{lem: NewCecileResult} in the following fashion: for the first sum, we employ the $L^{\infty}$ bound, and for the second sum we combine the $L^{1}$ bound together with the decreasing property of $\mathcal{F}_x(\al,\be)$. After some routine computations, it is possible to see that $E=o(x)$, completing the proof of Theorem \ref{th: n and rev(n) squarefree}.
\section{Squarefree integers with missing digits}
Given an integer $b\ge 2$, and $a_0\in \{0,\ldots ,b-1\}$, we recall that
\[
\cs{M}_b =\Big\{ \sum_{j=0}^{k}n_j b^j \; : \; n_j\in \{0,\ldots ,b-1\}\setminus \{a_0\}, \, k\ge 0\Big\}
\]
denotes the set of integers in base $b$ without $a_0$ in their expansion in base $b$. We also let $F_x(t)$ denote the normalized Fourier transform associated to $\cs{M}_b$. We summarize the results we need from Maynard \cite{MaynardDigitsOfPrimes, Maynard} in the following lemma:
\begin{lem}
Assume that $x=b^k$ and $y=b^l$ for some $k,l\in\bb{Z}^{+}$. 
\begin{enumerate}[{\rm (i)}]
  \item (\cite{Maynard}, (2.3)) If $x\le y$, then $F_y(t)\le F_x(t)$ for all $t\in [0,1]$.
  \item (\cite{MaynardDigitsOfPrimes}, Lemma 8.4) Let $d\gt 1$ be an integer with $d\lt x^{\f13}$, and $(d,b)=1$. Then, for all $a\lt d$ with $(a,d)=1$, we have an $L^{\infty}$ bound of the form
      \[
      F_x\Big(\frac{a}{d}\Big)\ll \exp\Big(-c_b \frac{\log{x}}{\log{d}}\Big)
      \]
  for some positive constant $c_b$ that depends only on $b$. 
  \item (\cite{Maynard}, Lemma 10.2) We have an $L^{1}$ bound of the form 
  \[
  \int_{0}^{1}F_x(t)\dt \ll \Big(\frac{\lambda}{b}\Big)^k=\frac{1}{x^{1-\frac{\log{\lambda}}{\log{b}}}},
  \]
  where $\la$ is the largest eigenvalue of the $b^3\times b^3$ matrix $M$ given by 
  \[
  M_{ij}\defeq \begin{cases}
                 G(t_1,t_2,t_3,t_4), {}& \mbox{if } i-1=t_2+t_3b+t_4b^2, j-1=t_1+t_2b+t_3b^2 \\ {}&\mbox{for some } t_1,t_2,t_3,t_4\in \{0,\ldots, b-1\} \\
                 0,{}& \mbox{otherwise},
               \end{cases}
  \]
  where 
  \[
  G(t_1,t_2,t_3,t_4)\defeq \sup_{|u|\le b^{-4}}\frac{1}{b-1}\Big|\frac{\me(\sum_{j=1}^{4}t_j b^{-j} + bu)-1}{\me(\sum_{j=1}^{4}t_j b^{-j-1} + u)-1}-\me(\sum_{j=1}^{4}a_0 t_j b^{-j-1} + a_0u) \Big|.
  \]
\end{enumerate}
\end{lem}
Therefore, $\cs{A}$ satisfies the assumptions of Theorem \ref{th: Main Theorem}, with the slightly more restrictive condition that $(d,b)=1$. As we have seen in previous sections, this only changes the asymptotic by a factor of $\prod_{p|b}(1-1/p^2)^{-1}$. More precisely, if we let $\al\defeq 1-\frac{\log{\lambda}}{\log{b}}$, a completely analogous argument as in the previous sections, shows that if $\al\gt 1/k$, then $\cs{A}$ contains infinitely many $k^{th}$ powerfree integers. Moreover 
\[
  \sum_{\substack{n\in\cs{M}_b^{*}(x) \\n \text{ is } k^{\text{th}} \text{ powerfree}}}1\sim \frac{\#\cs{M}_b^{*}(x)}{\zeta(k)}\prod_{p|b}\Big(1-\frac{1}{p^k}\Big)^{-1},
  \]
where $\cs{A}^{*}=\{a\in \cs{A} \; : \; (a,b)=1\}$. The following table shows the approximate values of $\al$ depending on the base $b$ and the excluded digit\footnote{We adapted the Mathematica file from Maynard \cite{Maynard}, and adjusted it accordingly for lower bases $b$. Our Mathematica file can be found as an ancillary file on arXiv:2504.08502.} $a_0$:
\begin{figure}[h!]
\begin{tabular}{|c||*{7}{c|}}
\hline
\multicolumn{1}{|c||}{$a_0 \backslash b$} & 3 & 4 & 5 & 6 & 7 & 8 & 9 \\
\hline\hline
0 & 0.37837 & 0.51110 & 0.57643 & 0.61599 & 0.64284 & 0.66249 & 0.67762 \\
\hline
1 & 0.36285 & 0.45387 & 0.52250 & 0.57233 & 0.60633 & 0.62852 & 0.64559 \\
\hline
2 & 0.37837 & 0.45387 & 0.55152 & 0.56422 & 0.59963 & 0.61732 & 0.63896 \\
\hline
3 &         & 0.51110 & 0.52250 & 0.56422 & 0.61955 & 0.61711 & 0.63511 \\
\hline
4 &         &         & 0.57643 & 0.57233 & 0.59963 & 0.61711 & 0.65576 \\
\hline
5 &         &         &         & 0.61599 & 0.60633 & 0.61732 & 0.63511 \\
\hline
6 &         &         &         &         & 0.64284 & 0.62852 & 0.63896 \\
\hline
7 &         &         &         &         &         & 0.66249 & 0.64559 \\
\hline
8 &         &         &         &         &         &         & 0.67762 \\
\hline
\end{tabular}
\caption{Values of $\alpha$ accurate to 5 decimals.}
\end{figure}

From the table and our previous discussion, Theorem \ref{th: SquarefreeMissingDigits} follows immediately. \\

\textbf{Acknowledgments} I would like to thank C\'ecile Dartyge for introducing me to the problem of powerfree palindromes. I am grateful to Daniel Johnston for pointing out the results from Bhomik and Suzuki regarding the $L^{1}$ and $L^{\infty}$ estimates for the Fourier transform of the set associated to reversible integers. I would also like to thank my PhD supervisor, Lola Thompson, for her helpful suggestions throughout the course of writing this paper. I am grateful to David Hokken and Berend Ringeling for their insightful discussions on the decreasing nature of $f_{2l+1}(t)$. 
\printbibliography
\end{document}